\documentclass[11pt]{article}
\usepackage{amsfonts,amsmath, amssymb,latexsym,mathrsfs}
\usepackage[OT2,OT1]{fontenc}
\usepackage[all,cmtip]{xy}
\def\Z{{\mathbb Z}}

\def\odd{{\rm odd}}
\def\sat{{\rm sat}}
\def\even{{\rm even}}

\def\ss{{\rm ss}}
\def\GL{{\rm GL}}

\def\Gal{{\rm Gal}}

\def\Jac{{\rm Jac}}

\def\rk{{\rm rk}}
\def\an{{\rm an}}
\def\P{{\mathbb P}}

\def\Inv{{\rm Inv}}

\def\res{{\rm res}}
\def\cris{{\rm cris}}

\def\R{{\mathbb R}}
\def\F{{\mathbb F}}

\def\Q{{\mathbb Q}}

\def\CE{{\mathcal E}}
\def\CF{{\mathcal F}}
\def\H{{\mathcal H}}

\def\C{{\mathcal C}}

\def\Z{{\mathbb Z}}
\def\P{{\mathbb P}}
\def\F{{\mathbb F}}
\def\Q{{\mathbb Q}}
\def\C{{\mathbb C}}

\def\H{{\mathcal H}}

\def\Sel{{\mathrm{Sel}}}
\def\Ql{{\Q_\ell}}
\def\Qp{{\Q_p}}
\def\Zp{{\Z_p}}
\def\Gal{{\mathrm{Gal}}}
\def\bQ{\bar\Q}

\def\grp{\mathfrak{p}}

\def\isoarrow{\stackrel{\sim}{\rightarrow}}

\def\Zl{{\Z_\ell}}
\def\ord{\mathrm{ord}}
\def\eps{\varepsilon}
\def\CE{\mathcal{E}}

\def\qed{{$\Box$ \vspace{2 ex}}}

\newtheorem{theorem}{Theorem}
\newtheorem{corollary}[theorem]{Corollary}
\newtheorem{lemma}[theorem]{Lemma}
\newtheorem{proposition}[theorem]{Proposition}
\newtheorem{remark}[theorem]{Remark}
\newenvironment{proof}{\noindent {\bf Proof:}}{$\Box$ \vspace{2 ex}}

\title{A positive proportion of elliptic curves over $\Q$ have rank one}

\author{Manjul Bhargava and Christopher Skinner}

\begin{document}
\maketitle

\begin{abstract}
  We prove that, when all elliptic curves over $\Q$ are ordered by
  naive height, a positive proportion have both algebraic and analytic
  rank one.  It follows that the average rank and the average analytic
  rank of elliptic curves are both strictly positive.
\end{abstract}

\tableofcontents 

\section{Introduction}

Any elliptic curve $E$ over $\Q$
has a unique Weierstrass model of the form
$E_{A,B}:y^2=x^3+Ax+B$ with $A,B \in \Z$ and such that for all primes
$\ell$:\, $\ell^6 \nmid B$ whenever $\ell^4 \mid A$. The (naive) {\it height}
of the elliptic curve
is then defined by
$$H(E)=H(E_{A,B}):= \max\{4|A^3|,27B^2\}.$$
When all elliptic curves
are ordered by their heights (or indeed, by their discriminants or
conductors), it is a well-known conjecture of Goldfeld~\cite{G1} and
of Katz and Sarnak~\cite{KS} that the average rank of all elliptic
curves is 1/2.  In fact, they conjectured that $50\%$ of all elliptic
curves should have rank 0 and $50\%$ should have rank 1, with a
neglible proportion having rank $\geq 2$.  However, as far as proofs,
it was not previously known whether a positive proportion of curves
have rank 0~or~1, whether the $\liminf$ of the average rank is $>0$,
or whether the $\limsup$ of the average rank is $<\infty$!  (See
\cite{BMSW} for a nice survey.)

In this direction, in recent papers~\cite{BS3,BS5} (using crucial
input also from the works \cite{DD} and~\cite{SU}) it was shown that,
when all elliptic curves are ordered by height, a positive proportion
have rank~0, and in fact analytic rank 0; moreover, the $\limsup$ of
the average rank of all elliptic curves is finite, and in fact less
than 1.

The purpose of this article is to prove the analogous positive proportion result for rank~1:

\begin{theorem}\label{rk1}
When all elliptic curves over $\Q$ are ordered by height, a positive proportion have rank~$1$.
\end{theorem}
In other words, Theorem~\ref{rk1} states that a positive proportion of
elliptic curves over $\Q$ have infinitely many rational points.

\vspace{.01in}
In fact, we prove that a positive proportion of elliptic curves over $\Q$ have both algebraic and analytic rank~$1$. More precisely, for any elliptic curve $E$ over $\Q$,
let us use $\rk(E)$ and $\rk_\an(E)$ to denote the algebraic and analytic rank of $E$, respectively.
Then we prove
\begin{eqnarray}\label{precisethm}
{\liminf_{X\to\infty}}\,
\frac
{\#\{E :
\rk(E)=\rk_\an(E)=1\mbox{ and } H(E)<X \}}
{\#\{E : H(E)<X\}}&>&0.
\end{eqnarray}
In order to keep the arguments as transparent as possible, in this
paper we simply prove positivity in (\ref{precisethm}), and we have
not tried to maximize the lower bound on this limit obtainable via the
methods employed to prove Theorem~\ref{rk1}.

As a consequence we also obtain, for the first time, a positive lower
bound on the $\liminf$ of the average rank of elliptic curves:

\begin{corollary}\label{avgrk}
When all elliptic curves over $\Q$ are ordered by height, the average rank and the average analytic rank are both strictly positive.
\end{corollary}

Finally, we note that
Theorem \ref{rk1}
also implies, via rank 1 curves, that a positive proportion of
elliptic curves satisfy the Birch and Swinnerton-Dyer rank conjecture.
The corresponding result via rank~0 curves was proven in \cite{BS3}.

In the next section, we describe the method of proof, which uses in an
essential way a number of recent developments in the arithmetic of
elliptic curves: the Gross--Zagier formula (in the general form proved
by Yuan, Zhang, and Zhang~\cite{YZZ-GZ}); a $p$-adic variant of the
Gross--Zagier formula (due to Bertolini, Darmon, and Prasanna~\cite{BDP}
and Brooks~\cite{Hunter}); work on an Iwasawa Main Conjecture
for $\GL_2$ (by Wan~\cite{Wan-U31}, building on the earlier
work~\cite{SU}); a converse to a theorem of Gross, Zagier,
and Kolyvagin~\cite{Skinner-GZ}; and the determination of the average
orders of $p$-Selmer groups of elliptic curves (particularly for $p=5$)~\cite{BS5}.

\section{Method of proof}\label{method}

To prove Theorem~\ref{rk1} (and Corollary~\ref{avgrk}),
we first establish two convenient sets of $p$-adic criteria that can be used to deduce the existence of rank one curves.  The first of these results provides sufficient conditions for an elliptic curve $E$ to have (algebraic and analytic) rank~one; the second provides analogous criteria for the quadratic twist $E^D$ of $E$ by an imaginary quadratic field $\Q(\sqrt{D})$
(also called the $D$-twist of $E$) to have rank one.

\begin{theorem}\label{crit1}
Let $E/\Q$ be an elliptic curve and let $p\geq 5$ be a prime. Suppose that:
\begin{itemize}
\item[\rm (a)] the conductor $N_E$ of $E$ is squarefree and has at least two odd prime factors;
\item[\rm (b)] $E$ has good, ordinary reduction at $p$;
\item[\rm (c)] $E[p]$ is an irreducible $\Gal(\bQ/\Q)$-module;
\item[\rm (d)] $\Sel_p(E)\cong \Z/p\Z$;
\item[\rm (e)] the image of $\Sel_p(E)$ in $E(\Qp)/pE(\Qp)$ under the restriction map at $p$
is not contained in the image of $E(\Qp)[p]$.
\end{itemize}
Then the algebraic rank and analytic rank of $E$ are both equal to $1$.
\end{theorem}

\vspace{0in}
\begin{theorem}\label{crit2}
Let $E/\Q$ be an elliptic curve and let $p\geq 5$ be a prime. Let $K/\Q$ be an imaginary quadratic
field of odd discriminant $D$ such that $2$ and $p$ split in $K$. Suppose that:
\begin{itemize}
\item[\rm (a)] the conductor $N_E$ of $E$ is squarefree with at least two odd prime factors,
 and $(N_E,D)=1$;
\item[\rm (b)] $E$ has good, ordinary reduction at $p$;
\item[\rm (c)] $E[p]$ is an irreducible $\Gal(\bQ/\Q)$-module and ramified at some odd prime
$q$ that is inert in $K$;
\item[\rm (d)] $\Sel_p(E)=0$ and $\Sel_p(E^D)\cong \Z/p\Z$;
\item[\rm (e)] The image of $\Sel_p(E^D)$ in $E^D(\Qp)/pE^D(\Qp)$ under the restriction map at $p$
is not contained in the image of $E^D(\Qp)[p]$.
\end{itemize}
Then the algebraic rank and analytic rank of $E^D$ are both equal to~$1$.
\end{theorem}

Both of these results will be proven in Section 3 using Theorem B of
\cite{Skinner-GZ}, which gives sufficient $p$-adic criteria for an
elliptic curve to have algebraic rank and analytic rank $1$ over an
imaginary quadratic field. The proof of \cite[Thm.~B]{Skinner-GZ} in
turn relies crucially on \cite{BDP}, \cite{Hunter}, and
\cite{Wan-U31}.

We employ Theorems \ref{crit1} and \ref{crit2} with $p=5$, $K=\Q[\sqrt{-39}]$ (so $D=-39$), $q=7$,
and $E$ belonging to a suitable ``large'' family of elliptic curves. (A large family of elliptic curves is one that is defined by congruence conditions and consists of a positive proportion of all elliptic curves; see Section~\ref{Equi} for a more precise definition.)  The large family we use is the set $\CF$ of elliptic curves $E=E_{A,B}$
such that
\begin{itemize}
\item $2^3\,||\,A$ and $2^4||B$;
\item $\Delta(A,B):=-4A^3-27B^2$ equals $2^8\Delta_1(A,B)$ with $\Delta_1(A,B)$
positive and squarefree (and necessarily odd), $(\Delta(A,B),5\cdot 39)=1$, and $\Delta(A,B)$ is a square modulo $39$;
\item $E$ has non-split multiplicative reduction at $7$;
\item $E$ has good, ordinary reduction at $5$.
\end{itemize}
For these curves, the discriminant of $E$ is $16\Delta(A,B)=2^{12}\Delta_1(A,B)$ and the conductor is
just $\Delta_1(A,B)$, which is squarefree and odd.

Since $7\mid \Delta_1(A,B)$ but $7$ is not a square modulo 39,
$\Delta_1(A,B)$ must be divisible by at least two odd primes. Hence
with our above choices of $p=5$ and $K=\Q(\sqrt{-39})$, conditions (a)
and (b) of Theorems~\ref{crit1} and \ref{crit2} hold for all
of
the curves in $\CF$.
Since all the curves $E$ in $\CF$ are semistable with non-split reduction at $7$,
$E[5]$ is an irreducible $\Gal(\bQ/\Q)$-module
and ramified at $7$, which is inert in $K=\Q[\sqrt{-39}]$; hence the conditions (c) of
Theorems \ref{crit1} and \ref{crit2} also hold, with $q=7$ in the case
of Theorem \ref{crit2}, for all curves in $\CF$.

To show that a positive proportion of elliptic curves also satisfy condition (d) of either
Theorem~\ref{crit1} or \ref{crit2}, we
use the following result on the average order of the $5$-Selmer group in large families of
elliptic curves, obtained in \cite{BS5}:

\begin{theorem}\label{thBS}{\bf (\cite[Thm.~31]{BS5})}
When elliptic curves $E$ over $\Q$ in any large family are ordered by height, the average order of the
$5$-Selmer group $\Sel_5(E)$ is equal to $6$.
\end{theorem}

Theorem \ref{thBS} guarantees the existence of many curves with 5-Selmer rank 0 or~1; however, this
alone is not sufficient to guarantee the existence of curves satisfying either of the conditions~(d)
in Theorems~\ref{crit1} and \ref{crit2}.   In order to deduce positive proportion statements for
rank 1 curves, we make use of information
regarding the distribution of the {\it parity} of
the $5$-Selmer ranks of
these curves and their $-39$-twists.  First, we note that the definition of $\CF$ implies that for all
$E\in\CF$, the curves $E$ and $E_{-39}$ have root numbers with opposite signs.
It follows that all curves $E\in\CF$ have the property that either $E$ or $E_{-39}$ have root number
$-1$; in particular, for all elliptic curves of height at most~$X$, either at least
$50\%$ of curves in $\CF$ have root number $-1$ or at least $50\%$ of their $-39$-twists have root
number $-1$.
The following theorem of Dokchitser--Dokchitser~\cite{DD}
(see also Nekov\'a\v{r}~\cite{Nekovar}) then allows us to relate these root numbers with $p$-Selmer ranks.

\begin{theorem}[Dokchitser--Dokchitser]\label{thDD}
Let $E$ be an elliptic curve over $\Q$ and let $p$ be any prime.
Let $s_p(E)$ and $t_p(E)$ denote the rank of the $p$-Selmer group of $E$ and the rank of $E(\Q)[p]$, respectively.
Then the quantity $r_p(E):=s_p(E)-t_p(E)$ is even $($resp.~odd$\,)$ if and only if the root
number of $E$ is~$+1$ $($resp.~$-1)$.
\end{theorem}

Let $\CF^D$ be the set of $D$-twists of curves in $\CF$ (recall that $D=-39$).
Since $E[5]$, and hence $E^D[5]$, is irreducible for all $E\in\CF$,
by Theorem~\ref{thDD}
we see that either at least half of the curves in $\CF$ of height at most~$X$
have odd $5$-Selmer rank, or at least half of the curves in $\CF^D$ of height at most $39^5 X$
have odd $5$-Selmer rank.  Of these odd $5$-Selmer rank curves in $\CF$ or $\CF^D$, not all could be
of $5$-Selmer rank $\geq 3$,
or the average size of the $5$-Selmer group of the curves in $\CF$ and $\CF^D$ could not be $6$;
indeed, at least a proportion of 19/20
of these odd $5$-Selmer rank curves must have 5-Selmer rank 1. This is enough to deduce that at least a proportion of 19/20 of the elliptic curves in $\CF$
also satisfy (d) for {\it either} Theorem~\ref{crit1} {\it or} Theorem~\ref{crit2}.

To show that a positive proportion of these latter curves also satisfy
the corresponding condition (e) of Theorem~\ref{crit1} or~\ref{crit2},
we need the following equidistribution result, whose proof is discussed in
Section~4:

\begin{theorem}\label{equi}
Let $F$ be a large family of elliptic curves over $\Q$, and let $E=E_{A_0,B_0}$ be any elliptic curve
in $F$. There exists a $5$-adic neighborhood $W\subset \Z_5^2\backslash\{\Delta=0\}$ of $(A_0,B_0)$
such that the large subfamily $F(W)$ of $F$ containing all the curves $E_{A,B}$ in $F$
with $(A,B)\in W$ has the property that:
\begin{itemize}
\item[\rm (i)]
For each $E'\in F(W)$,
$E'(\Q_5)/5E'(\Q_5)$ is naturally identified with $E(\Q_5)/5E(\Q_5)$, and
via this identification the image of the natural map $E'(\Q_5)[5]\rightarrow
E'(\Q_5)/5E'(\Q_5)$ is identified with that of $E(\Q_5)[5]\rightarrow E(\Q_5)/5E(\Q_5)$.

\vspace{.1in}

\item[\rm (ii)]
When the elliptic curves $E'\in F(W)$ are ordered by height,
the images
of the non-identity $5$-Selmer elements
under the natural restriction map
\begin{equation*}
\Sel_5(E')\stackrel{\res}{\to} E'(\Q_5)/5E'(\Q_5) = E(\Q_5)/5E(\Q_5)
\end{equation*}
are equidistributed in $E(\Q_5)/5E(\Q_5)$.
\end{itemize}
\end{theorem}

We use Theorem~\ref{equi} to deduce that a
proportion of at least 1/2 of the elliptic curves in~$\CF$, which automatically satisfy (a)--(c),
also satisfy conditions (d) and (e) of either Theorem~\ref{crit1} or Theorem~\ref{crit2};
%, {\it and} further also satisfy the corresponding condition (e).
Theorem~\ref{rk1} then follows.
Details of this latter argument are discussed in Section 5.

We note that an analogue of the equidistribution result of
Theorem~\ref{equi} for elements of the 2-Selmer groups of Jacobians of odd
degree hyperelliptic curves $y^2=x^{2g+1}+\cdots$ of genus $g$ over
$\Q$ was established in \cite{BG}.  This equidistribution result was
used in the work of Poonen and Stoll~\cite{PS} to show that a positive
proportion of odd degree hyperelliptic curves of genus $g\geq 3$ have {\it
  no rational points} other than the obvious one at infinity.  It is
interesting that, in contrast, we use here the above equidistribution result
for 5-Selmer elements to prove the {\it existence of a non-trivial
  rational point} on a positive proportion of elliptic curves.

\section{$p$-adic criteria for an elliptic curve over $\Q$ to have rank one}

In this section, we prove Theorems 3 and 4, which provide criteria for
an elliptic curve $E$ or certain quadratic twists $E^D$ of $E$ to have
both algebraic and analytic rank 1; the criteria, in particular,
involve the $p$-Selmer groups of $E$ or $E^D$.  The proofs involve
showing that the hypotheses of these theorems satisfy those of Theorem
B of \cite{Skinner-GZ}.

\subsection{Selmer groups}\label{Selmer}

Let $E$ be an elliptic curve with conductor denoted $N_E$ and let $p$ be an odd prime.
Let $\bQ$ be an algebraic closure of $\Q$. For each prime $\ell$, let $\bQ_\ell$ be an algebraic closure of
$\Q_\ell$ and fix an embedding $\bQ\hookrightarrow \bQ_\ell$; the latter
realizes $G_{\Ql}=\Gal(\bQ_\ell/\Ql)$ as a decomposition subgroup for $\ell$ in $G_\Q=\Gal(\bQ/\Q)$.

Let $p$ be a prime. Recall that the $p^r$-Selmer group of $E$ is
$$
\Sel_{p^r}(E) = \ker\{H^1(\Q,E[p^r])\stackrel{\res}{\rightarrow}\prod_{\ell} H^1(\Ql,E(\bQ_\ell))\}
$$
and that the $p^\infty$-Selmer group of $E$ is
$$
\Sel_{p^\infty}(E) = \varinjlim_{r} \Sel_{p^r}(E) \subseteq H^1(\Q,E[p^\infty]).
$$
The local conditions on classes in $\Sel_{p^r}(E)$ (resp.~$\Sel_{p^\infty}(E)$), can also
be expressed as the restriction at each prime $\ell$ being
in the subgroup $E(\Ql)/p^rE(\Ql)\hookrightarrow H^1(\Ql,E[p^r])$ (resp.~$E(\Ql)\otimes\Qp/\Zp\hookrightarrow
H^1(\Ql,E[p^\infty]$), where the injection is just the usual Kummer map.

Let $T=T_pE$, $V=T\otimes_\Zp\Qp$, and $A=V/T=E[p^\infty]$ (the last identification being
given by $(x_n)\otimes\frac{1}{p^m}\mapsto x_m$).
The subgroup $E(\Ql)\otimes\Qp/\Zp\hookrightarrow H^1(\Ql,A)$ is also the image $H^1_f(\Ql,A)$
of $H^1_f(\Ql,V)$ in $H^1(\Ql,A)$, where
$H^1_f(\Qp,V)=\ker\{H^1(\Qp,V)\rightarrow H^1(\Qp,B_{\cris}\otimes_\Qp V)\}\cong\Qp$, where
$B_{\cris}$ is the ring of crystalline periods,
and $H^1_f(\Ql,V) = H^1(\F_\ell,V^{I_\ell})$ if $\ell\neq p$.
For $\ell\neq p$, $H^1_f(\Ql,A)$ (which is in fact $0$) is contained in $H^1(\F_\ell,A^{I_\ell})$ with
finite index in general
and with equality if $E$ has good reduction at $\ell$.
In particular, if $S$ is a finite set of primes containing all those that divide $pN_E$ and if $G_{\Q,S}$ is
the Galois group of the maximal extension of $\Q$ unramified outside $S$, then
$\Sel_{p^\infty}(E)\subset H^1(G_{\Q,S},A)$ consists of those classes with restriction to
$H^1(\Ql,A)$ belonging to $H^1_f(\Ql,A)$ for all $\ell\in S$.
As the image of $H^1(G_{\Q,S},V)$ in
$H^1(G_{\Q,S},A)$ is the maximal divisible subgroup (and has finite index), it follows that
the maximal divisible subgroup of $\Sel_{p^\infty}(E)$ is the image in
$H^1(\Q,A)$ of the characteristic zero Bloch--Kato Selmer group
$$
H^1_f(\Q,V) = \ker\{H^1(\Q,V)\stackrel{\res}{\rightarrow}\prod_\ell H^1(\Ql,V)/H^1_f(\Ql,V)\}.
$$
Also,
$$
\Sel_{p^\infty}(E)\twoheadrightarrow E(\Qp)\otimes\Qp/\Zp  \ \iff \
H^1_f(\Q,V)\twoheadrightarrow H^1_f(\Qp,V).
$$
Properties of the $p^\infty$-Selmer group $\Sel_{p^\infty}(E)$ and of
the Bloch--Kato Selmer group $H^1_f(\Q,V)$ can sometimes be deduced
from knowledge of just the $p$-Selmer group $\Sel_p(E)$, as we do in
the following lemma.

\begin{lemma}\label{Selmerlemma}
Suppose $E$ has good reduction at $p$ and $E(\Q)[p]=0$.
\begin{itemize}
\item[\rm (i)] If $\Sel_p(E)=0$, then $\Sel_{p^\infty}(E)=0$ and $H^1_f(\Q,V)=0$.
\item[\rm (ii)] The $\F_p$-dimension of $\Sel_p(E)$ is even $($resp.~odd$\,)$ if and only if
$w(E)=+1$ $($resp.~$w(E)=-1)$, where $w(E)$ is the root number of $E$.
\item[\rm (iii)] If $\Sel_p(E)\cong \Z/p\Z$ then $\Sel_{p^\infty}(E)\cong \Qp/\Zp$,
and if furthermore the image of $\Sel_p(E)$ in $E(\Qp)/pE(\Qp)$ is
not contained in the image of $E(\Qp)[p]$, then the restriction map
$\Sel_{p^\infty}\stackrel{\res}{\to} E(\Qp)\otimes\Qp/\Zp$ is an isomorphism; in particular,
$H^1_f(\Q,V)\isoarrow H^1_f(\Qp,V)$.
\end{itemize}
\end{lemma}

\begin{proof}
  Since $E(\Q)[p]=0$, $\Sel_{p}(E) = \Sel_{p^\infty}(E)[p]$. Part (i)
  is then immediate, and Part (ii) is a special case of the result
  of Dokchitser and Dokchitser stated in Theorem \ref{thDD}.

Cassels proved that $\Sel_{p^\infty}(E) \cong (\Qp/\Zp)^r \oplus F\oplus F$ for some finite group $F$,
so if $\Sel_{p^\infty}(E)[p]=\Sel_p(E)\cong\Zp/p\Zp$, then it follows that $r=1$ and $F=0$ and hence that
$\Sel_{p^\infty}(E)\cong \Qp/\Zp$.
As $E$ has good reduction at $p$, the reduction map induces an injection
$E(\Qp)[p^\infty]\hookrightarrow E(\F_p)[p^\infty]$. By the Riemann Hypothesis for $E$,
$p^2$ does not divide the order of $E(\F_p)$, so
$E(\Qp)[p]=E(\Qp)[p^\infty]$. From
the exact sequence
$$
0\rightarrow E(\Qp)[p^\infty]/pE(\Qp)[p^\infty] \rightarrow
E(\Qp)/pE(\Qp) \rightarrow (E(\Qp)\otimes\Qp/\Zp)[p]\rightarrow 0
$$
it then follows that if the image of the restriction map $\Sel_p(E)\rightarrow E(\Qp)/pE(\Qp)$ is
not contained in the image of $E(\Qp)[p]$ (= the image of $E(\Qp)[p^\infty]/pE(\Qp)[p^\infty]$), then
$\Sel_p(E)$ maps isomorphically onto $(E(\Qp)\otimes\Qp/\Zp)[p]\cong\Z/p\Z$.
That the restriction map $\Sel_{p^\infty}(E)\rightarrow E(\Qp)\otimes\Qp/\Zp \cong \Qp/\Zp$
is an isomorphism then follows from this injectivity at the level of $p$-torsion. Part (iii) follows.
\hfill\end{proof}

\subsection{Proofs of Theorems \ref{crit1} and \ref{crit2}}

Let $E$ be an elliptic curve over $\Q$. Let $K/\Q$ be an imaginary quadratic field
with discriminant denoted $D$, and suppose $p\geq 5$. Theorem B of \cite{Skinner-GZ}
asserts that if
\begin{itemize}
 \item[(I)] $E$ has good, ordinary reduction at $p$;
\item[(II)] $E[p]$ is an irreducible $G_\Q$-module and ramified at some odd prime $q\neq p$ that is
inert in $K$;
\item[(III)] both $2$ and $p$ split in $K$;
\item[(IV)] $(D,N_E)=1$;
\item[(V)] $\dim_\Qp H^1_f(K,V)=1$ and the restriction $H^1_f(K,V)\rightarrow
\prod_{\grp|p}H^1_f(K_\grp,V)$ is an injection,
\end{itemize}
then $E(K)$ has rank $1$ and $\ord_{s=1}L(E/K,s)=1$.\pagebreak

We recall that \cite[Thm.~B]{Skinner-GZ} is proved by showing that the formal logarithm
$\log_\grp P_K$ of a suitable Heegner point $P_K\in E(K)$
for a prime $\grp\mid p$ of $K$ does not vanish.
The proof of this is via Iwasawa Theory. Under hypotheses (I)--(IV) there exists
a $p$-adic $L$-function $L_\grp(E/K,\chi)$ that is a function of certain
anticyclotomic Hecke characters $\chi$ and that $p$-adically interpolates the algebraic parts
of the $L$-values $L(E,\chi,1) = L(V\otimes\sigma_\chi,0)$ for those $\chi$
having infinity-type $z^{-n}\bar z^n$ with $n>0$, where $\sigma_\chi$ is the
$p$-adic avatar of $\chi$. The connection with $\log_\grp P_K$ comes in a
remarkable formula for the value of this $p$-adic $L$-function at the trivial
character (which does not belong to the set of characters $\chi$ for which $L_{\grp}(E/K,\chi)$
is interpolating the value of a complex $L$-function) that was proved by
Bertolini, Darmon, and Prasanna~\cite[Main~Thm.]{BDP} and
Brooks~\cite[Thm.~IX.11]{Hunter} (see also \cite[\S 7.3]{Hunter}):
$$
L_\grp(E/K,1) = (1-a_p(E)+p)^2(\log_\grp P_K)^2.
$$

The Main Conjecture of Iwasawa Theory, as formulated by Greenberg, would identify
$L_\grp(E/K,\chi)$ as the generator of the characteristic ideal of a certain $p$-adic
Selmer group, one consequence of which would be
$$
L_\grp(E/K,1)=0\implies H^1_\grp(K,V)=0,
$$
where $H^1_\grp(K,V)\subset H^1(K,V)$ 
consists of those classes that are trivial at the primes of $K$ other than $\mathfrak{p}$
and are unrestricted at $\mathfrak{p}$.
Under the hypotheses (I)--(IV), 
Wan \cite[Thm.~1.1]{Wan-U31}
has proved enough in the direction of this conjecture to deduce this implication.
Furthermore, the hypothesis~(V) implies $H^1_\grp(K,V)=0$, hence
$L_\grp(E/K,1)\neq 0$, and so $\log_\grp P_K$ is non-zero. In particular,
$P_K\in E(K)$ is non-torsion. From (V) it then follows that the rank of $E(K)$ is 1,
and it follows from the general Gross--Zagier formula proved by Yuan--Zhang--Zhang~\cite{YZZ-GZ} that
$\ord_{s=1}L(E/K,s)=1$.

\bigskip

\noindent{\bf Proof of Theorem \ref{crit1}:}
By hypothesis, $E$ has good, ordinary reduction at $p$ and $E[p]$ is an irreducible $G_\Q$-module.
From hypotheses (d) and (e) and Lemma \ref{Selmerlemma} it then follows that
the root number of $E$ is $w(E)=-1$,
$\Sel_{p^\infty}(E)\cong\Qp/\Zp$, and $H^1_f(\Q,V)\isoarrow H^1_f(\Qp,V)$.
It also follows from Ribet's level-lowering result \cite{Ribet-Serre}
that $E[p]$ is ramified at some odd prime $q\neq p$
(otherwise, the Galois representation $E[p]$ would arise
from some cuspidal eigenform of weight $2$ and level $1$, of which there are none).
Let $K/\Q$ be an imaginary quadratic field such that
\begin{itemize}
\item the discriminant $D$ of $K$ is odd;
\item $2$ and $p$ split in $K$;
\item both the prime $q$ and one other odd prime divisor of $N_E$ are inert in $K$ and all other
prime divisors of $N_E$ split in $K$ (this is possible as $N_E$ is assumed to have at least two odd prime
divisors), in which case we have $w(E^D)=w(E)\chi_D(-N_E)= +1$ (see,
e.g., \cite[\S3.10]{Rohrlich});
\item $L(E^D,1)\neq 0$.
\end{itemize}
Such a quadratic field $K$ exists by a theorem of Friedberg and
Hoffstein \cite[Thm.~B]{FriedbergHoffstein}.  Here, $E^D$ is the
$D$-twist of $E$ and $\chi_D$ is the quadratic Dirichlet character of
conductor $D$ associated with $K$. We also use $\chi_D$ to denote the
quadratic character of $G_\Q$ having kernel $\Gal(\bQ/K)$.

We now appeal to Theorem B of \cite{Skinner-GZ} by verifying that 
hypotheses (I)--(V) hold for $E$ and~$K$.
That (I)--(IV) hold is immediate.
Since $L(E^D,1)\neq 0$, by \cite[Thm.~14.2]{Kato} or \cite[Cor.~B]{Kolyvagin},
we have $H^1_f(\Q,V_pE^D) = H^1_f(\Q,V\otimes\chi_D)=0$. In particular, $E^D(\Q)$ is finite. Then
$$
H^1_f(K,V)\cong H^1_f(\Q,V)\oplus H^1_f(\Q,V\otimes\chi_D) \cong H^1_f(\Q,V) \isoarrow
H^1_f(\Qp,V)\hookrightarrow \prod_{\grp|p} H^1_f(K_\grp,V),
$$
whence (V) also holds.  Therefore, $E(K)$ has rank 1 and $\ord_{s=1}L(E/K,s)=1$.
Since the rank of $E(K)$ is the sum of the ranks of $E(\Q)$ and $E^D(\Q)$ and since $E^D(\Q)$ is
finite, $E(\Q)$ must have rank 1.
As $L(E/K,s) = L(E,s)L(E^D,s)$ and $L(E^D,1)\neq 0$,
$\ord_{s=1}L(E/K,s) =1$ implies $\ord_{s=1}L(E,s)=1$.
\hfill\qed

\noindent{\bf Proof of Theorem \ref{crit2}:}
We show that (I)--(V) hold. Note that (I)--(IV) are immediate from the hypotheses of Theorem \ref{crit2}.
From Lemma \ref{Selmerlemma} it follows that
$\Sel_{p^\infty}(E)=0$ (so $E(\Q)$ is finite
and $H^1_f(\Q,V)=0$), $w(E)=+1$, $w(E^D)=-1$, $\Sel_{p^\infty}(E^D)\cong\Qp/\Zp$,  and
$H^1(\Q,V\otimes\chi_D)\isoarrow H^1_f(\Qp,V\otimes\chi_D)$. It follows that
$$
H^1_f(K,V)\cong H^1_f(\Q,V\otimes\chi_D) \isoarrow
H^1_f(\Qp,V\otimes\chi_D)\hookrightarrow \prod_{\grp|p} H^1_f(K_\grp,V),
$$
whence (V) also holds. We then conclude, as in the proof of Theorem \ref{crit1}
but reversing the roles of $E$ and $E^D$,
that $E^D(\Q)$ has rank 1 and $\ord_{s=1}L(E^D,s)=1$.
\hfill\qed

\section{Equidistribution of Selmer elements}
\label{Equi}

In this section, we prove Theorem~\ref{equi}, namely, that for any
``large'' family $F$ of elliptic curves~$E$ over $\Q$ lying in a
sufficiently small $\nu$-adic disc so that all $E(\Q_\nu)/5E(\Q_\nu)$
for $E\in F$ are naturally identified, the non-identity 
elements of the 5-Selmer group become equidistributed in 
$E(\Q_\nu)/5E(\Q_\nu)$.  We also make precise what we mean by
``naturally identified''.

\subsection{Counting Selmer elements}

We begin by recalling from \cite{BS5} what is meant by a large family of elliptic curves.
For each prime~$\ell$, let
$\Sigma_\ell$ be a closed subset of $\{(A,B)\in\Z_\ell^2 :
\Delta(A,B):=-4A^3-27B^2\neq 0\}$ with boundary of measure $0$.
To such a collection $\Sigma=(\Sigma_\ell)_\ell$,
we associate the set $F_\Sigma$ of elliptic curves over $\Q$, where
$E_{A,B}\in F_\Sigma$ if and only if $(A,B)\in\Sigma_\ell$ for all~$\ell$.
We then say that $F_\Sigma$ is a family of elliptic curves over~$\Q$ 
that is {\it defined by congruence conditions}.
We can also impose ``congruence conditions at infinity'' on $F_\Sigma$, by
insisting that an elliptic curve $E_{A,B}$ belongs to $F_\Sigma$ if
and only if $(A,B)$ belongs to $\Sigma_\infty$, where $\Sigma_\infty$
consists of all $(A,B)$ with $\Delta(A,B)$ positive, or negative, or either.

If $F$ is a family of elliptic curves over $\Q$ defined by congruence
conditions, then let $\Inv(F)$ denote the set $\{(A,B):E_{A,B}\in
F\}$. We define $\Inv_p(F)$ to be the set of those elements $(A,B)$ in
the $p$-adic closure of $\Inv(F)\subset\Z_p^2$ such that
$\Delta(A,B)\neq 0$.  We define $\Inv_\infty(F)$ to be
$\{(A,B)\in\R^2:\Delta(A,B)>0\}$, $\{(A,B)\in\R^2:\Delta(A,B)<0\}$, or
$\{(A,B)\in\R^2:\Delta(A,B)\neq 0\}$ in accordance with whether $F$
contains only curves of positive discriminant, negative discriminant,
or both, respectively.  Then a family $F$ of elliptic curves defined
by congruence conditions is said to be {\it large} if, for all
sufficiently large primes $\ell$, the set $\Inv_\ell(F)$ contains all
pairs $(A,B)\in\Z_\ell^2$ such that $\ell^2\nmid\Delta(A,B)$. For
example, the family of all elliptic curves is large, as is any
family of elliptic curves $E_{A,B}$ defined by finitely many congruence
conditions on $A$ and $B$.  The family
of all semistable elliptic curves is also large. Any large family
makes up a positive proportion of all elliptic curves~\cite[Thm.~3.17]{BS2}.

To explain the proof of Theorem~\ref{equi}, we briefly outline the
strategy of the proof from~\cite{BS5} of Theorem~\ref{thBS}.  The
proof uses the 
representation $V=5\otimes\wedge^25$ 
of quintuples $(A_1,\ldots,A_5)$ of
$5\times 5$ skew-symmetric matrices.  The group
$\GL_5\times\GL_5$ 
acts naturally on $V$, by
\begin{equation}\label{gl5action}
(g_1,g_2)\cdot
(A_1,A_2,A_3,A_4,A_5):=(g_1A_1g_1^t,g_1A_2g_1^t,g_1A_3g_1^t,g_1A_4g_1^t,g_1A_5g_1^t)\cdot g_2^t.
\end{equation}
Let the {\it determinant} of an element
$(g_1,g_2)\in\GL_5\times\GL_5$ by defined by $\det(g_1,g_2):=(\det
g_1)^2\det g_2$, and let $G$ denote the algebraic group  
\begin{equation}\label{eqg}
G:=\{(g_1,g_2)\in\GL_5\times\GL_5:\det(g_1,g_2)=1\}/\{(\lambda I_5,\lambda^{-2} I_5)\},\end{equation}
where $I_5$ denotes the identity element of $\GL_5$ and $\lambda\in
\mathbb{G}_{\rm m}$.  Then the action of
$\GL_5\times\GL_5$ on $V$ induces an action of $G$ on $V$.
The ring of polynomial invariants for the action of $G(\C)$ on $V(\C)$
turns out to have two independent generators,
having degrees 20 and 30, which we may denote by $A(v)$ and $B(v)$,
respectively.

Now let $K$ be a field of characteristic prime to 2, 3, and 5.  If
$v=(A_1,\ldots,A_5)\in V(K)$, then let
$v(t_1,\ldots,t_5):=A_1t_1+\cdots+A_5t_5$ denote the corresponding
matrix of linear forms in $t_1,\ldots,t_5$.  Then the determinant of
$v(t_1,\ldots,t_5)$ is zero,
since the determinant of any odd-dimensional skew-symmetric matrix is
zero.  So instead we %take
consider its principal $4\times 4$ sub-Pfaffians (i.e., canonical
square-roots of the principal $4\times 4$ minors of the matrix
$v(t_1,\ldots,t_5)$).  This yields five quadrics in five variables,
which (whenever $\Delta(v) = \Delta(A(v),B(v)):=-4A(v)^3-27B(v)^2\neq 0$) cut out a smooth
genus one curve $C(v)$ in $\P^4$ over $K$ that is embedded by a complete linear
system of degree 5.  We have chosen our generators $A=A(v)$ and $B=B(v)$
of the invariant ring so that the Jacobian $E(v)$ of this genus one curve is the
elliptic curve given by
\begin{equation*}\label{jacformula}
E_{A,B}:y^2=x^3+Ax+B.
\end{equation*}
The discriminant $\Delta(v)$ on $V(K)$ (whose nonvanishing detects
stable orbits on $V(K)$) thus coincides with the discriminant of
the associated Weierstrass equation (\ref{jacformula}) of $E(v)$.
Conversely, given any genus one curve $C$ over $\Q$ in $\P^4$ embedded
by a compete linear system of degree 5 and with Jacobian~$E_{A,B}$,
there exists an
element of $V(\Q)$ having invariants 
$A$ and $B$, unique up to the
action of $G(\Q)$, that cuts out the curve $C$ in $\P^4$ in this way.
Moreover, if the curve $C$ has a point at every place of~$\Q$ (i.e.,
if $C$ is {\it locally soluble}), then there exists an element of
$V(\Z)$ with invariants 
that cuts out $C$~(see~\cite[Thm.~2.1]{Fishermin})!

We say that an element $v\in V(\Z)$ (or $V(\Q)$) is {\it locally
  soluble} if the curve $C(v)$ has a point locally at every place of
$\Q$.  We say that $v$ is ${\it soluble}$ if $C(v)$ has a rational
point.  The $G(\Q)$-orbits of locally soluble elements in $V(\Q)$
having invariants $A$ and $B$ with $\Delta(A,B)\neq 0$ are then in
natural bijection with the elements of the 5-Selmer group
$\Sel_5(E_{A,B})$ of $E_{A,B}$, and each such locally soluble
$G(\Q)$-orbit contains an element in $V(\Z)$.  The analogous
statements remain true when $\Q$ and $\Z$ are replaced by $\Q_\ell$
and $\Z_\ell$, and the 5-Selmer group of $E_{A,B}(\Q)$ is replaced by
$E_{A,B}(\Q_\ell)/5E_{A,B}(\Q_\ell)$. The correspondence of a soluble
$v\in V(\Q_\ell)$ with an element of the local $5$-Selmer group can be
made explicit through a 5-cover $\phi_v:C(v)\rightarrow E_{A,B}$ given
in terms of certain covariants for the action of $G$ on $V$. 
All these latter facts and much more information on this
representation can be found in the works of
Fisher~\cite{Fisher1,Fisher3,Fisher15} (see also~\cite{BhHo,BS5}).

In~\cite{BS5}, the $G(\Q)$-equivalence classes of locally soluble
elements $v\in V(\Z)$ having bounded height and $\Jac(C(v))\in F$ are
counted asymptotically.  By dividing through by the asymptotic number
of elliptic curves over $\Q$ of bounded height in $F$, the average
size of the 5-Selmer group of all elliptic curves in $F$ is obtained,
as in Theorem~\ref{thBS}.  Specifically, if $F=F_\Sigma$, then in
\cite[Prop.\ 33 and \S4.2]{BS5} (using \cite[Theorem~29]{BS5}),
it is first proved that
\begin{equation}\label{eqthsec5}
  \displaystyle\lim_{X\to\infty}\frac{\displaystyle\sum_{E\in F, H(E)<X}(\#\Sel_5(E)-1)}{\displaystyle\sum_{E\in F, H(E)<X}1}=\tau(G)\,
\frac{M_\infty(V,F;X)}{M_\infty(F;X)} \displaystyle\prod_\ell
\frac{M_\ell(V,F)}{M_\ell(F)}
\end{equation}
where
\begin{eqnarray*}\label{eqmpvf}
M_\ell(V,F)&:=&\displaystyle\int_{(A,B)\in \Inv_\ell(F)}\frac{1}{\#E_{A,B}(\Q_\ell)[5]}
\sum_{\sigma\in E_{A,B}(\Q_\ell)/5E_{A,B}(\Q_\ell)} \!1 \;\;\;
dAdB, \\
M_\ell(F)&:=&\displaystyle\int_{(A,B)\in \Inv_\ell(F)}
dAdB, \\[.1in]
    M_\infty(V,F;X)&:=&\displaystyle\int_{\substack{(A,B)\in\Inv_\infty(F)\\H(A,B)<X}}\,\,\displaystyle\frac{1}{\# E_{A,B}(\R)[5]}\,\,\!\:\sum_{\sigma\in E_{A,B}(\R)/5E_{A,B}(\R)}\,1\,\;\;\;dAdB,\\
        M_\infty(F;X)&:=&\displaystyle\int_{\substack{(A,B)\in\Inv_\infty(F)\\H(A,B)<X}}dAdB,
\end{eqnarray*}
and $\tau(G)$ denotes the Tamagawa number of $G$ as an algebraic group
over $\Q$. 
The right hand side of Equation (\ref{eqthsec5}) then simplifies to $\tau(G)=5$.

\subsection{Proof of Theorem~\ref{equi}}

We may extend the arguments of \S4.1 to also prove  equidistribution of
$5$-Selmer elements in $E(\Q_\nu)/5E(\Q_\nu)$.  We first use the
representation of $G(\Q_\nu)$ on $V(\Q_\nu)$ to naturally identify the 
quotients $E(v)(\Q_\nu)/5E(v)(\Q_\nu)$ corresponding to elements $v$ in
sufficiently small neighborhoods in~$V(\Q_\nu)$.  With this
identification, we then obtain the following version of Theorem~\ref{equi}:

\begin{theorem}\label{equi2}Fix a place $\nu$ of $\Q$.
Let $F=F_\Sigma$ be a large family of elliptic curves $E$ such that
\begin{itemize}
\item[{\rm (a)}]
the cardinality of $E(\Q_\nu)/5E(\Q_\nu)$ is a constant $k$ for all $E$ in $F;$ and
\vspace{.1in}
\item[{\rm (b)}] the set
\begin{equation*}\label{updef}
U_\nu(F) := \{\mbox{soluble\ elements\ in $V(\Z_\nu)$ having invariants $(A,B)$
s.t. $(A,B)\in\Sigma_\nu$}\}
\end{equation*}
  can be partitioned into $k$ open sets $\Omega_1,\ldots,\Omega_k$
  such that$:$
\begin{itemize}
\item[{\rm (i)}]
for all $i$, if two elements in $\Omega_i$ have the same invariants
$A,B$, then they are $G(\Q_\nu)$-equivalent$;$ and
\item[{\rm (ii)}]
 for all $i\neq j$, %we have
$(G(\Q_\nu)\cdot \Omega_i)\cap (G(\Q_\nu)\cdot\Omega_j)=\emptyset$.
\end{itemize}
\end{itemize}
Then for $E\in F$, the elements of $E(\Q_\nu)/5E(\Q_\nu)$
are in bijection with the sets $\Omega_i$.
$($In particular, the groups $E(\Q_\nu)/5E(\Q_\nu)$
are identified for all $E$ in $F$.$)$
When the elliptic curves $E$ in $F$ are ordered by height, the images of the
nonidentity $5$-Selmer elements under the restriction map
\begin{equation*}
\Sel_5(E)\stackrel{\res}{\to} E(\Q_\nu)/5E(\Q_\nu)\leftrightarrow \{\Omega_1,\ldots,\Omega_k\}
\end{equation*} are equidistributed.
\end{theorem}

To prove Theorem~\ref{equi2},
on the right hand side of Equation~(\ref{eqthsec5})
(which comes from \cite[Prop.\ 33 and \S4.2]{BS5}),
we replace the sum over all $\sigma$ in $E_{A,B}(\Q_\nu)/5E_{A,B}(\Q_\nu)$ in the expression
for $M_\nu(V,F)$ with
the corresponding sum over $\sigma$ in any subset $S\subset
E_{A,B}(\Q_\nu)/5E_{A,B}(\Q_\nu)$.  By property (b), we are still
counting elements in a weighted subset of $V(\Z)$ defined by an
``acceptable'' set of congruence conditions, so \cite[Theorem~29]{BS5}
again applies.  This gives us the average number of nonidentity
5-Selmer elements that map to $S$, and we see that the result is
proportional to the size of $S$, proving Theorem~\ref{equi2}.

\begin{remark}{\em
Note that the same argument also allows one to show equidistribution of non-identity 5-Selmer elements in
$\prod_{\nu\in S}E(\Q_\nu)/5E(\Q_\nu)$ for any finite set~$S$ of places of $\Q$, provided that our
large family $F$ of elliptic curves lies in the intersection of sufficiently small $\nu$-adic discs ($\nu\in S$)
so that both (a) and (b) are satisfied for all $\nu\in S$.}
\end{remark}

To deduce Theorem~\ref{equi}, we combine Theorem~\ref{equi2} with two
propositions.  The first shows that if the elliptic curves in a large
family $F$ all have invariants in a sufficiently small $\nu$-adic
neighborhood, then the local $p$-Selmer groups are naturally
identified so that the images of the torsion groups $E(\Q_\nu)[p]$ are
also identified. The second proposition states that for a family of
elliptic curves with invariants~$A,B$ in a similarly small neighborhood,
conditions (a) and (b) of Theorem~\ref{equi2} hold. Comparing these
two propositions and their proofs then shows that for small enough
neighborhoods, the identification of local $5$-Selmer groups from the
first of these propositions agrees with that coming from
Theorem~\ref{equi2}, from which Theorem~\ref{equi} follows.

\begin{proposition}\label{localprop} Let $p\geq 5$ and $\ell$ be primes.
Let $E\subset \P^2$ be an elliptic curve over $\Ql$ given by the Weierstrass equation
\begin{equation}\label{W1}
y^2 + a_1xy + a_3 = x^3 + a_2x^2+a_4x+a_6, \ \ a_i\in\Ql.
\end{equation}
There exists $h=h(E)\geq 2$ such that if $E'\subset\P^2$ is another elliptic curve
over $\Ql$ given by a Weierstrass equation
\begin{equation}\label{W2}
y^2 + a_1'xy + a_3' = x^3 + a_2'x^2+a_4'x+a_6', \ \ a_i'\in\Ql,
\end{equation}
with
\begin{equation}\label{W3}
 |a_i-a_i'|_\ell \leq \ell^{-h},
\end{equation}
then $E(\Ql)/pE(\Ql)$ and $E'(\Ql)/pE'(\Ql)$ are
identified so that
\begin{itemize}
\item[$(i)$] if $P=(x,y)\in E(\Ql)$ and $P'=(x,y)\in E'(\Ql)$ are such that $|x-x'|_\ell \leq \ell^{-h}$,
$|y-y'|_\ell\leq \ell^{-h}$, then the images of $P$ and $P'$ are identified;
\item[$(ii)$] the natural images of $E(\Ql)[p]\rightarrow E(\Ql)/pE(\Ql)$ and
$E'(\Ql)[p]\rightarrow E'(\Ql)/pE'(\Ql)$ are identified.
\end{itemize}
\end{proposition}

\begin{proof}
  We first assume that \eqref{W1} is a minimal Weierstrass
  equation. In particular, $a_i\in \Zl$, Equation~\eqref{W1} defines a closed
  $\Zl$-subscheme of $\P^2$ that we also denote by $E$, and the open
  subscheme~$E_0$ obtained by removing the singular points on the
  closed fiber of $E$ is a group scheme over $\Zl$
  ($E_0(\Zl)=E_0(\Ql)$ is just the subgroup of points of $E(\Ql)$
  having nonsingular reduction modulo~$\ell$).  Suppose also that $E$
  does not have split multiplicative reduction. Then the component
  group $E(\Ql)/E_0(\Ql)$ has order at most~$4$, hence, as $p\geq 5$,
  $E_0(\Ql)/pE_0(\Ql)\isoarrow E(\Ql)/pE(\Ql)$. As $E_0(\Ql) =
  E_0(\Zl)$, it is easy to see that the reduction modulo $\ell^2$ map
  (modulo $\ell$ suffices if $\ell\neq p$) induces an isomorphism
  $E_0(\Ql)/pE_0(\Ql)\isoarrow
  E_0(\Zl/\ell^2\Zl)/pE_0(\Zl/\ell^2\Zl)$.

  Suppose \eqref{W3} holds.  If $h\geq 2$ is sufficiently large (in
  terms of the $a_i$'s), then \eqref{W2} will also be a minimal
  equation and $E'$ will have the same reduction type as $E$, and so,
  since $E_0$ and $E_0'$ have the same reduction modulo $\ell^2$,
  there are natural identifications
$$
E(\Ql)/pE(\Ql) \cong E_0(\Zl/\ell^2\Zl)/pE_0(\Zl/\ell^2\Zl) = E_0'(\Zl/\ell^2\Zl)/pE_0'(\Zl/\ell^2\Zl) = E'(\Ql)/pE'(\Ql).
$$
Then it is clear that $(i)$ also holds. If $E$ (and hence $E'$) has
split multiplicative reduction, then a similar identification for
which $(i)$ also holds can be deduced from the Tate uniformization of
$E(\Ql)$ and $E'(\Ql)$. As this case is not needed in this paper, we
omit the details.

For a general Weierstrass equation \eqref{W1}, we note that if $h$ is
large enough, then any transformation $(x,y)\mapsto (u^2x+r,
u^3y+u^2sx+t)$, with $r,s,t,u\in\Q_\ell$, that takes \eqref{W1} to a minimal
equation also takes a Weierstrass equation \eqref{W2} (provided that
\eqref{W3} holds) to a minimal equation. Furthermore, if $h$ is
sufficiently large (with respect to a fixed such transformation), then
the coefficients of the resulting minimal equations will be close
enough that the preceding arguments yield the desired identification
and that $(i)$ holds.

We now show that $(ii)$ also holds for $h$ large enough.  Let $h_0$ be
such that for a Weierstrass equation \eqref{W2}, if
$|a_i-a_i'|_\ell\leq \ell^{-h_0}$ then $E(\Ql)/pE(\Ql)$ and
$E'(\Ql)/pE'(\Ql)$ are naturally identified so that if $P=(x,y)\in
E(\Ql)$ and $P'=(x',y')\in E'(\Ql)$ are points with
$|x-x'|_\ell,|y-y'|_\ell\leq \ell^{-h_0}$, then the images of $P$ and
$P'$ are identified (we have just proved the existence of such an
$h_0$).

The $x$-ordinates of the $p^2-1$ non-trivial $p$-torsion points
$P_i=(x_i,y_i)$ and $P_i'=(x_i',y_i')$, $i=1,\ldots,p^2-1$, of
$E(\bQ_\ell)$ and $E'(\bQ_\ell)$, respectively, are the roots of
polynomials of degree $(p^2-1)/2$ (the division polynomials denoted
$\psi_p$ in \cite{Silverman}) whose coefficients are given by
universal polynomials over $\Z$ in the coefficients of the Weierstrass
equations \eqref{W1} and \eqref{W2}. It then follows, say from
Krasner's lemma, that if the Weierstrass equations \eqref{W1} and
\eqref{W2} are sufficiently close (that is, if $|a_i-a_i'|_\ell <
\ell^{-h_1}$ with $h_1>h_0$ sufficiently large) then the coefficients
of these division polynomials are so close that the $P_i$ and $P_i'$
can be ordered so that for each $i$, $P_i$ and $P_i'$ are defined over
the same extension of~$\Ql$ and satisfy $|x_i-x_i'|_\ell\leq
\ell^{-h_0}$ and $|y_i-y_i'|_\ell\leq \ell^{-h_0}$, and so, by the
choice of $h_0$, the points in $E(\Ql)[p]$ and $E'(\Ql)[p]$ are
identified in $E(\Ql)/pE(\Ql)=E'(\Ql)/pE'(\Ql)$. Thus the conclusions
of the proposition hold with $h(E)=h_1$.  \hfill\end{proof}

\noindent

As explained in the proof of this proposition, if the Weierstrass equation
\eqref{W1} is minimal and~$E$ does not have split multiplicative reduction at $\ell$, then
Part $(i)$ of the proposition holds with~$h=2$ in \eqref{W3}.

\begin{proposition}\label{nbdexistence}
Fix a place $\nu$ of $\Q$.   For any given $(A,B)\in
\Z_\nu^{2}\setminus\{\Delta=0\}$, there exists a $\nu$-adic neighborhood $W$ of
$(A_0,B_0)=(A,B)$ in
$\Z_\nu^{2}$ such that the corresponding $($large$)$ family $F=F(W)$, consisting of all elliptic
curves $E_{A,B}$ with $(A,B)\in W$, satisfies both %hypotheses
{\rm (a)} and~{\rm (b)} of Theorem~$\ref{equi2}$.
\end{proposition}

\begin{proof}
The proposition is trivial for $\nu=\infty$, as we may let $W$ equal the set of all $(A,B)$
having discriminant positive, negative, or either.
If $\nu$ is a finite prime $p$, let $k$ be the cardinality of
$E_{A_0,B_0}(\Q_p)/5E_{A_0,B_0}(\Q_p)$.  Then the set $Y$ of soluble
elements in the inverse image of $(A,B)$
under the map $\pi:V(\Z_p) \to \Z_p^{2}$, given by sending $v$ to the invariants $(A(v),B(v))$,
is the disjoint union of $k$ nonempty compact sets $Y_1,\ldots,Y_k$,
namely, the $k$\, $G(\Q_p)$-equivalence classes in $V(\Z_p)$
comprising~$Y$.

Let
$Z_1,\ldots,Z_k\subset V(\Z_p)\setminus\{\Delta=0\}$ be disjoint
neighborhoods of $Y_1,\ldots,Y_k$, respectively, in $V(\Z_p)$ such
that each $Z_i$ consists of soluble elements and is the union of
$G(\Q_p)$-equivalence classes in $V(\Z_p)$.
Such~$Z_i$ can be constructed by noting that if $\varepsilon$ is
sufficiently small, then the $\varepsilon$-neighborhoods
$B_{\varepsilon}(Y_i)$ of the $Y_i$'s are disjoint and consist only of
elements that have nonzero discriminant and are soluble.
The set $\{ g \in G(\Q_p) \,|\, g B_\varepsilon(Y_i)\cap
V(\Z_p)\neq \emptyset \}$
is then compact.  Indeed, for a single stable element $v\in V(\Z_p)$
(i.e., $v$ has nonzero discriminant $\Delta(v)$),
the set
\begin{equation}\label{star}
\{ g \in G(\Q_p)\,|\, g\cdot v \in V(\Z_p) \}
\end{equation}
is compact and constant in a neighborhood of $v$.  This may be seen
from the decomposition of $G(\Q_p)$ as $G(\Z_p)T(\Q_p)G(\Z_p)$, where $T(\Q_p)$
is the torus consisting of the diagonal elements of $G(\Q_p)$; since
the set of elements of $T(\Q_p)$ taking a stable element $v'\in
V(\Z_p)$ to $V(\Z_p)$ is compact and constant in a neighborhood of
$v'$, it follows that (\ref{star}) is also compact and constant in a
neighborhood of $v$.

The compactness of
\begin{equation*}\label{star2}
\{ g \in G(\Q_p) \,|\, gS \cap V(\Z_p)\neq \emptyset \}
\end{equation*}
now follows for any compact set $S$ of stable elements (by covering $S$
with neighborhoods where (\ref{star}) is constant, and then taking a
finite subcover).  Hence $(G(\Q_p)\cdot B_{\varepsilon}(Y_i))\cap
V(\Z_p)$ is both open and compact, and so is a bounded distance away
from $Y_j$ for all $j\neq i$. By shrinking $\varepsilon$ if necessary,
we can then ensure that $(G(\Q_p)\cdot B_{\varepsilon}(Y_i))\cap
B_{\varepsilon}(Y_j)=\emptyset$ for all $i\neq j$, and therefore
$(G(\Q_p)\cdot B_{\varepsilon}(Y_i))\cap (G(\Q_p)\cdot
B_{\varepsilon}(Y_j))\cap V(\Z_p)=\emptyset$ for all $i\neq j$.  We then set
$Z_i=(G(\Q_p)\cdot B_{\varepsilon}(Y_i))\cap V(\Z_p)$.

Let $W'=\{(A,B)\in\Z_p^{2}\,|\,(A,B)\in\cap_i\,\pi(Z_i)\}$.
Since~$\pi$ is an open mapping on $V(\Z_p)\setminus\{\Delta=0\}$, we
see that $W'$ is an open set in $\Z_p^{2}$ containing $(A,B)$.  Let
$W\subset W'$ be an open neighborhood of $(A,B)$ small enough so that
for all elliptic curves $E=E_{A,B}$ with $(A,B)\in W$, we have
$\#(E(\Q_p)/5E(\Q_p))=k$. Such a neighborhood $W$ exists because the
size of $E(\Q_p)/5E(\Q_p)$ is locally constant (see Proposition
\ref{localprop}).
Then $F(W)$ satisfies both (a) and (b), with
$\Omega_i=Z_i\cap\pi^{-1}(\{
%(A,B)\,|\,
(A,B)\in W\})$.
\hfill\end{proof}

We now complete the proof of Theorem~\ref{equi}.
We set $\nu=p=5$ in Proposition \ref{nbdexistence}.  
For $i=1,\dots,k$ in Theorem~\ref{equi2}, 
fix $v_i\in Y_i$ and $P_i\in C(v_i)(\Q_5)$ (so $Y_i$ is
identified with the image of $\phi_v(P_i)\in E_{A_0,B_0}(\Q_5)/5E_{A_0,B_0}(\Q_5))$.
If $\varepsilon$ is small enough, then for any $v\in
B_\varepsilon(Y_i)$ we have:
\begin{itemize}
\item $(A,B)=(A(v),B(v))$ are close
enough $5$-adically to $(A,B)$ so 
that for $E = E_{A_0,B_0}$ and $E' = E_{A,B}$, 
the conclusions of Proposition~\ref{localprop} hold;
\item the quadratic equations defining $C(v)$ are close enough
  $5$-adically to those defining $C(v_i)$ to ensure that there is a
  point $P'\in C(v)(\Q_5)$ close enough to $P_i$ so that the
  hypotheses of Proposition~\ref{localprop}(i) hold for
  $\phi_{v_i}(P_i)\in E(\Q_5)$ and $\phi_v(P)\in E'(\Q_5)$; 
  therefore, the images of $\phi_{v_i}(P_i)$ and $\phi_v(P)$ agree
  under the identification $E(\Q_5)/5E(\Q_5)= E'(\Q_5)/5E'(\Q_5)$
  given by Proposition~\ref{localprop}.
\end{itemize}
It follows that the identification of $E(\Q_5)/5E(\Q_5)$ with
$E'(\Q_5)/5E'(\Q_5)$ given by the sets $\Omega_i$ is the same as the
identification in Proposition \ref{localprop}.  In particular, by
Proposition~\ref{localprop}(ii), the $W$ in
Proposition \ref{nbdexistence} can be taken so that
the identification of local Selmer
groups in Theorem~\ref{equi2} also identifies the images of the
$p$-torsion subgroups. Theorem \ref{equi} then follows from the
conclusions of Theorem \ref{equi2} for the large set $F(W)$.

\begin{remark}
{\em 
The analogues of Theorems~\ref{equi} and
\ref{equi2} for the equidistribution of elements in 
2-, 3-, and 4-Selmer groups (instead of the
5-Selmer group) may be proven by analogous arguments, using the 
results in~\cite{BS2}, \cite{BS3},
and \cite{BS4}, respectively (instead of \cite{BS5}). 
}
\end{remark}

\section{Counting curves: Proof of Theorem \ref{rk1}}
We recall the definition of the set $\CF$
from Section \ref{method}: $\CF$ consists of those elliptic curves $E=E_{A,B}$ such that
\begin{itemize}
\item $2^3||A$ and $2^4||B$;
\item $\Delta(A,B):=-4A^3-27B^2$ equals $2^8\Delta_1(A,B)$ with $\Delta_1(A,B)$
positive and squarefree (and necessarily odd), $(\Delta(A,B),5\cdot 39)=1$, and $\Delta(A,B)$ is a square modulo $39$;
\item $E$ has non-split multiplicative reduction at $7$;
\item $E$ has good, ordinary reduction at $5$.
\end{itemize}
The discriminant of such an $E$ is $16\Delta(A,B)$ and its conductor is just $\Delta_1(A,B)$,
which is squarefree and odd. In particular, $E$ is a semistable curve.
The set $\CF$ is a large family of elliptic curves defined by congruence conditions in the sense
of \cite[\S3]{BS2} (see also Subsection~4.1).

We begin by verifying that all elliptic curves $E\in\CF$ satisfy
properties (a)--(c) in both Theorems~\ref{crit1} and \ref{crit2} (with
$p=5$, $K=\Q(\sqrt{-39})$, and $q=7$):

\begin{lemma}\label{irredlem}
Let $E\in\CF$. The $G_\Q$-module $E[5]$ is irreducible and ramified
at every prime factor of the conductor $N_E$, so in particular at $7$.
\end{lemma}

\begin{proof} If $E[5]$ were reducible, then its semisimplification would be a sum of two characters:
$E[5]^{\ss}\cong \F_5(\chi\omega)\oplus\F_5(\chi^{-1})$, with $\omega$ the mod $5$ cyclotomic character.
As $E$ has semistable reduction, the conductor of $E[5]^{\ss}$ at a prime $\ell\neq 5$,
which divides the conductor of $E$, can be at most~$\ell$,
from which it follows that $\chi$ is unramified at all primes different from $5$. But since $E$ also
has good, ordinary reduction at $5$, it must be that either $\chi^{-1}$ or $\chi\omega$ is unramified
at $5$ and so is unramified everywhere.
It then follows that either $\chi=1$ or $\chi=\omega^{-1}$,
whence $E[5]^{\ss} \cong \F_5(\omega)\oplus \F_5$.
Since $E$ is also assumed to have non-split, multiplicative reduction at $7$, one of the eigenvalues of
a Frobenius at 7 on $E[5]^{\ss}$ must be $-1$ modulo $5$, a contradiction
(as $7\not\equiv-1\;\mathrm{(mod}\; 5)$). It follows that $E[5]$ is
irreducible. The condition that $E[5]$ be unramified at a prime $\ell\neq 5$ of
multiplicative reduction is that the discriminant $\Delta_\ell$ of a minimal Weierstrass
model at $\ell$ satisfy $\ord_\ell(\Delta_\ell)\equiv 0\;\mathrm{(mod}\;5)$
\cite[Chap.~V, Prop.~6.1 \& Ex.~5.13]{Silverman2}. But the Weierstrass model $E_{A,B}$ of $E$
is clearly a minimal model at each odd prime $\ell$ since $\Delta_1(A,B)$ is squarefree by hypothesis,
and so $\ord_\ell(\Delta_\ell)=\ord_\ell(\Delta(A,B))=1$.

\hfill\end{proof}

We now turn to counting various families of elliptic curves, in order
to establish that a positive proportion of elliptic curves $E\in \CF$
also satisfy properties (d)--(e) in {\it either} Theorem~\ref{crit1}
{\it or}~\ref{crit2} (with the same choices of $p$ and $K$). 
We recall (\cite[Thm.~3.17]{BS2}):

\begin{lemma}\label{largecount} For any large set $F$ of elliptic curves,
there exists a constant $c(F)>0$ such that
$$
\#\{E\in F : H(E)< X\} = c(F)X^{5/6}+o(X^{5/6}).
$$
\end{lemma}
Thus, in particular, the elliptic curves in our large family $\CF$ have positive
density in the family of all elliptic curves over $\Q$, when ordered
by height.

Our aim now is to count the number of curves in $\CF$ that satisfy either
the hypotheses of Theorem~\ref{crit1} with $p=5$, or the hypotheses of 
Theorem~\ref{crit2} with $p=5$, $K=\Q[\sqrt{-39}]$ (so $D=-39$), and $q=7$.
Let
\begin{eqnarray*}
N(X) &=& \#\{E\in\CF : H(E)<X\}\\
N_i(X) &=& \#\{E\in \CF : H(E)<X \text{ and } \Sel_p(E)\cong(\Z/p\Z)^i\}\\
N_\even(X) &=& \#\{E\in\CF : H(E)<X \text{ and }
\Sel_p(E)\cong(\Z/p\Z)^{2j} \text{ for some $j$}\}\\
N_\odd(X) &=& \#\{E\in\CF : H(E)<X \text{ and }
\Sel_p(E)\cong(\Z/p\Z)^{2j+1} \text{ for some $j$}\}
\end{eqnarray*}
Then $N(X)=N_\even(X)+N_\odd(X) =\sum_{i=0}^\infty N_i(X)$.  Note
that, by 
%the result of Dokchitser and Dokchitser in
Lemma~\ref{Selmerlemma}(ii), we may also write
\begin{eqnarray*}
N_\even(X) &=& \#\{E\in\CF : H(E)<X \text{ and }
w(E)=1\}\\
N_\odd(X) &=& \#\{E\in\CF : H(E)<X \text{ and }
w(E)=-1\}.
\end{eqnarray*}
Similarly, let
$$N_i^D(X) = \#\{E\in\CF : H(E)<X \text{ and } \Sel_p(E^D)\cong (\Z/p\Z)^i\}.
$$
By Lemmas \ref{irredlem} and \ref{Selmerlemma}(ii), the curves counted by $N_0(X)$ and $N^D_1(X)$ (resp.\
by $N_1(X)$ and $N^D_0(X)$) are a subset of those counted by $N_\even(X)$ (resp.\ $N_\odd(X)$).

We also need to count the number $\widetilde N_1(X)$ of curves in $\CF$
with height $<X$ and $\Sel_p(E)\cong\Z/p\Z$ whose restriction to
$E(\Qp)/pE(\Qp)$ lies in the image of $E(\Qp)[p]$, and the number
$\widetilde N_1^D(X)$ of curves in $\CF$ with height $<X$ and
$\Sel_p(E^D)\cong \Z/p\Z$ whose restriction to $E^D(\Qp)/pE^D(\Qp)$
lies in the image of $E^D(\Qp)[p]$.

\begin{lemma}\label{badcountlemma}
We have
$$
\widetilde N_1(X) \leq \frac{1}{p-1}N(X) - (N_\even(X)-N_0(X))-(p+1)(N_\odd(X)-N_1(X)) + \widetilde\eps_1(X)$$
and
$$\;\;
\widetilde N_1^D(X) \leq \frac{1}{p-1}N(X) -(N_\odd(X)-N_0^D(X))-(p+1)(N_\even(X)-N_1^D(X)) + \widetilde\eps_1^D(X),
$$
where both $\widetilde\eps_1(X)$ and $\widetilde\eps_1^D(X)$ are $o(X^{5/6})$.
\end{lemma}

\begin{proof}  By Theorem~\ref{equi}, the large family $\CF$ can be partitioned into a finite union of large subfamilies
for each of which the groups
$E(\Qp)/pE(\Qp)$ as well as the images of $E(\Qp)[p]$ have been identified.
Furthermore, as each $E\in \CF$ has good, ordinary reduction at $p$,
we have that $\#E(\Qp)/pE(\Qp) = p\cdot\#E(\Qp)[p]$ is equal to $p$ or $p^2$.
By Theorem \ref{thBS}, Lemma \ref{largecount},
and the equidistribution result of Theorem \ref{equi} for the curves in each of these subfamilies,
the number of non-trivial Selmer elements in $\Sel_p(E)$ for some $E\in \CF$ with $H(E)<X$
that restrict to an element in the image of $E(\Qp)[p]$ is
$$\frac{\#E(\Qp)[p]}{\#E(\Qp)/pE(\Qp)}pN(X) + o(X^{5/6}) = N(X) + \widetilde\eps_1(X)$$
where $\widetilde\eps_1(X) = o(X^{5/6})$.

Now for $E\in\CF$, if $w(E)=1$ but $\Sel_p(E)\neq 0$, then
$\#\Sel_p(E)\geq p^2$ by Lemma \ref{Selmerlemma}(ii).  Since the
restriction map is a homomorphism, there are at least $p-1$
non-trivial elements in $\Sel_p(E)$ whose restriction lies in the
image of $E(\Qp)[p]$. Similarly, if $w(E)=-1$ but $\Sel_p(E)\neq 0$,
then $\#\Sel_p(E)\geq p^3$, and there are at least $p^2-1$ non-trivial
elements in $\Sel_p(E)$ with restriction in the image of
$E(\Qp)[p]$. It follows that
$$
\widetilde N_1(X)(p-1) + (p-1)(N_\even(X)-N_0(X)) + (p^2-1)(N_\odd(X)-N_1(X)) \leq N(X) + \widetilde\eps_1(X)
$$
where $\widetilde\eps_1(X)=o(X^{5/6})$, 
yielding the inequality for $\widetilde N_1(X)$ in the statement of the lemma.
An identical argument applies to $\widetilde N_1^D(X)$.
\hfill\end{proof}

Let $\CF_\sat\subset\CF$ be the subset of curves that satisfy the
hypotheses of Theorem \ref{crit1} with $p=5$, and let
$\CF_\sat^D\subset\CF$ be the subset of curves that satisfy the
hypotheses of Theorem \ref{crit2} with $p=5$, $K=\Q[\sqrt{-39}]$, and
$q=7$.  These are disjoint subsets of $\CF$.  Let 
\begin{equation*}
N_\sat(X) = \#\{E\in\CF_\sat : H(E)<X\} 
\end{equation*}
and similarly 
\begin{equation*}
N_\sat^D(X) = \#\{E\in \CF_\sat^D : H(E)<X\}.
\end{equation*}

\begin{proposition}\label{countprop}
We have
$$
N_\sat(X)+N_\sat^D(X) \geq N(X)\left(1-\frac{2}{p-1}\right)
+ \eps_\sat(X)
$$
with $\eps_\sat(X)=o(X^{5/6})$.
\end{proposition}

\begin{proof} The definition of $\CF$ together with Lemma \ref{irredlem}
and the observation that $7$ is not a square modulo $39$ (or the observation that there are no
elliptic curves of conductor $7$) shows that any curve in the
count $N_1(X)$ satisfies (a)--(d) of
Theorem \ref{crit1},
so $$N_\sat(X)= N_1(X)-\widetilde N_1(X).$$ Similarly,
the number of curves in $\CF$ of height $<X$ that satisfy (a)--(d) of
Theorem \ref{crit2} is at least $N^D_1(X)+N_0(X)-N_\even(X)$, so
$$N_\sat^D(X) \geq N^D_1(X)-\widetilde N^D_1(X) +N_0(X)-N_\even(X).$$ Thus, by Lemma \ref{badcountlemma},
\begin{equation*}\begin{split}
N_\sat(X) & \geq N_1(X)-\frac{N(X)}{p-1}+(N_\even(X)-N_0(X)) + (p+1)(N_\odd(X)-N_1(X)) -\widetilde\eps_1(X) \\
& \geq N_\odd(X) -\frac{N(X)}{p-1} + N_\even(X)-N_0(X) -\widetilde \eps_1(X),
\end{split}
\end{equation*}
and, similarly,
\begin{equation*}\begin{split}
N_\sat^D(X) & \geq N^D_1(X)-\widetilde N^D_1(X) + N_0(X)-N_\even(X) \\
& \geq N_\even(X) - \frac{N(X)}{p-1} + N_0(X)-N_\even(X) -\widetilde\eps_1^D(X).
\end{split}
\end{equation*}
Combining these two inequalities yields the lower bound in the proposition.
\hfill\end{proof}

\noindent{\bf Proof of Theorem 1:} The number of elliptic curves of height $<D^6X$ that have
both analytic and algebraic rank one is at least the number of curves $E\in\CF$ of height $<X$ satisfying
the hypotheses of Theorem \ref{crit1} with $p=5$ plus the number of elliptic curves $E^D$ with
$E\in\CF$ of height $<X$ (since $H(E^D)=D^6H(E))$)
satisfying the hypotheses of Theorem \ref{crit2} with $p=5$, $K=\Q[\sqrt{-39}]$,
and $q=7$---that is, at least $N_\sat(X)+N_\sat^D(X)$. Therefore,
\begin{equation}\label{lastlim}
{\liminf_{X\to\infty}}\,
\frac
{\#\{E :
\rk(E)=\rk_\an(E)=1\mbox{ and } H(E)< D^6X \}}
{\#\{E : H(E)< D^6X\}}
\geq \liminf_{X\to\infty} \frac{N_\sat(X)+N_\sat^D(X)}{\#\{E : H(E)< D^6X\}}.
\end{equation}
Let $\CE$ be the set of all elliptic curves over $\Q$.
As $\#\{E : H(E)<D^6X\} = c(\CE)D^5X^{5/6}+o(X^{5/6})$ for some $c(\CE)>0$ by
Lemma \ref{largecount}, it follows by Proposition \ref{countprop}
that the right hand side of (\ref{lastlim}) is at least 
$$\liminf_{X\to\infty}\frac{N(X)\bigl(1 -{\frac{2}{p-1}}\bigr)}{\#\{E : H(E)<D^6X\}}
= \frac{c(\CF)}{c(\CE)}\cdot\frac{1/2}{39^5}\;>\:0.
$$
%proving Theorem~\ref{rk1}.
\hfill\qed

\begin{remark}{\em For the counting argument given, it is crucial that
    we were able to work with $p=5$. For $p=3$ the upper bound on
    $\widetilde N_1(X)$ and $\widetilde N_1^D(X)$ would be
    $\frac{1}{2}N(X)$, which means we would have had to potentially
    exclude 50\% of the curves in $\CF$ from our counts due to the restriction
    conditions~(e) in Theorems~\ref{crit1} and \ref{crit2} (instead of
    only 25\% for $p=5$); however, we only expect 50\%
    of the curves in $\CF$ and $\CF^D$ to
    have root number $-1$. This is reflected in the lower bound for
    $N_\sat(X)+N_\sat^D(X)$ in Proposition \ref{countprop}, which
    would be trivial if $p=3$.}
\end{remark}

\subsection*{Acknowledgments}

We thank Arul Shankar and Xin Wan for many helpful conversations. 
The first named author was
supported in part by National Science Foundation Grant~DMS-1001828 and
a Simons Investigator Grant.  The second named author was supported in
part by National Science Foundation Grants~DMS-0701231 and
DMS-0758379.

\end{document}